\documentclass[12 pt,leqno]{article}

\usepackage{amsmath,url}
\usepackage{amssymb}
\usepackage[english]{babel}
\usepackage{amsthm}
\usepackage[latin1]{inputenc}
\usepackage{hyperref}
\usepackage{graphicx}
\usepackage{subfigure}
\usepackage{makeidx}

\newtheorem{teo}{Theorem}
\newtheorem{lemma}{Lemma}
\newtheorem{prop}{Proposition}

\newtheorem{cor}{Corollary}
\newtheorem{remark}{Remark}

\newenvironment{sistema}%
{\left\lbrace\begin{array}{@{}l@{}}}%
{\end{array}\right.}

{\left( \begin{array}{@{}l@{}}}%
{\end{array}\right)}

\DeclareMathOperator{\esssup}{ess\,sup}

\begin{document}

\title{\textbf{Coexistence of bounded and unbounded motions in a bouncing ball model}}
\author{\textbf{Stefano Marò} \\
\textit{\small{Dipartimento di Matematica - Università di Torino}}\\
\textit{\small{Via Carlo Alberto 10, 10123 Torino - Italy}}\\
\textit{\small{e-mail: stefano.maro@unito.it}}
}
\date{}

\maketitle

\begin{abstract}
We consider the model describing the vertical motion of a ball falling with constant acceleration on a wall and elastically reflected. The wall is supposed to move in the vertical direction according to a given periodic function $f$. We apply the Aubry-Mather theory to the generating function in order to prove the existence of bounded motions with prescribed mean time between the bounces. As the existence of unbounded motions is known, it is possible to find a class of functions $f$ that allow both bounded and unbounded motions.  
\end{abstract}

\section{Introduction}  

We consider the model of a ball bouncing on an infinitely heavy racket that is moving in the vertical direction according to a given regular periodic function $f(t)$. Moreover, we suppose that the gravity force is acting on the ball.

\begin{figure}[h]
\begin{center}
\includegraphics[scale=1.0]{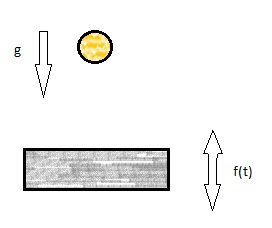}
\end{center}
 \end{figure}
 
This model can be formulated in terms of continuous or discrete dynamics. Actually, we can follow the continuous motion of the ball or just look at the sequence of impact times and velocities $(t_n,w_n)$ produced by successive impacts. The first approach leads to a differential equation and the second to a map $S(t_n,w_n)=(t_{n+1},w_{n+1})$ in the plane. In this context an unbounded motion has to be understood in the sense of the possibility of speeding up the ball, i.e. finding an orbit $(t_n^*,w_n^*)$ of $S$ such that $\lim_{n\to +\infty}w_n^*=+\infty$ or, equivalently $\lim_{n\to +\infty}(t_{n+1}^*-t_n^*)=+\infty$.
 
\smallskip 
\noindent The existence of unbounded motions was first proved by Pustyl'nikov \cite{pust}. Under some assumptions on $f(t)$, he used the discrete description to find an orbit $(t_n^*,w_n^*)$ corresponding to a motion in which every impact occurs when the racket is moving upwards. Moreover, he was also able to prove that the stable manifold around this orbit turns to be composed by unbounded orbits. More recently the model has been studied in \cite{ruiztorres}. We also mention the paper by Dolgopyat \cite{dolgo} dealing with non-gravitational forces and the paper by Kunze and Ortega \cite{kunzeortega2} concerning a non-periodic function $f(t)$. 

\smallskip
The aim of this paper is to prove that bounded motions with remarkable qualitative characteristics are possible as well. Indeed we will prove that for every real and sufficiently big number $\alpha$ there exists an invariant set of orbits with bounded velocity. In each of these invariant sets, there exists an orbit corresponding to a motion with mean time between the bounces coinciding with $\alpha$. Moreover, motions with the same mean time between the bounces can be ordered. If $\alpha=p/q$ is rational, the invariant set contains periodic points of period $q$ that correspond to $p$-periodic motions with $q$ bounces in a period. Between two consecutive periodic points there is an heteroclinic orbit. If $\alpha$ is irrational, then we have an alternative: either an invariant curve corresponding to quasi-periodic motions in the classical sense with frequencies the period of $f(t)$ and $\alpha$ or the invariant set is a Cantor. In the last case the motion is not quasi-periodic in the classical sense but displays a dynamics of Denjoy-type \cite{mathertop}. We stress the fact that both the alternatives are possible. Indeed the unbounded orbit breaks the invariant curves, and if $\dot{f}$ differs little from zero then all the motions of the ball are bounded and invariant curves appears \cite{pustsoviet}. It is worth mentioning that these two facts occur for different functions $f$. 

\smallskip
To prove the result we use the classical theory of Aubry-Mather \cite{aubry,matherams}. Precisely, after replacing the velocity by the energy, the map $(t_n,E_n)\mapsto(t_{n+1},E_{n+1})$ becomes symplectic and has an associated variational principle. This means that the sequence of successive impact times $(t_n)$ satisfies the second order difference equation
$$
\partial_2h(t_{n-1},t_n)+\partial_1h(t_n,t_{n+1})=0
$$
where $h=h(t_0,t_1)$ is the so-called generating function. A nice feature of this model is that the function $h$ can be computed explicitly. This was done in \cite{kunzeortega2} and we will employ it. Section 2 is dedicated to a brief discussion on the formulation of the model and the variational principle. In Section 3 we prove that the standard Aubry-Mather theory can be applied to our model. This will allow to solve the previous difference equation, whose solutions will give the motions described above. In fact there are some subtleties involved, because the general theory asks for generating function defined in the whole plane $(t_0,t_1)$, while our function $h$ is only defined in an half-plane $t_1-t_0\geq k$. Then we have to extend $h$ to the whole plane preserving the condition $\partial_{t_0t_1}h\leq\epsilon< 0$, that is crucial in Aubry-Mather theory. To achieve this we use a trick based on the D'Alambert formula for the wave equation
$$
\frac{\partial^2h}{\partial t_0\partial t_1}=p(t_0,t_1).
$$
The use of the Aubry-Mather theory gives more informations on the orbits that are summed up in Corollary \ref{cr}.

\smallskip
Finally, we remark that, joining the result of Pustil'nikov on unbounded orbit with ours, it is possible to get a class of motions of the racket that allow both bounded and unbounded motions of the ball.

%

\section{Statement of the problem}
Consider the model of a free falling ball with mass $1$ under the gravity force $g$ and let $z(t)$ be its vertical position. The plate is supposed to move according to a $C^4(\mathbb{R})$  $1$-periodic function $f(t)\leq z(t)$. We can consider a system of reference joined with the plate performing the change of variable $x(t)=z(t)-f(t)$. At an instant $\tau$ of impact, the change of velocity is assumed to be elastic. So we will consider the problem    
\begin{equation}\label{prob}
\begin{sistema}
\ddot{x}=-(g+\ddot{f}(t))\\
x(t)\geq0\\
x(\tau)=0\Rightarrow\dot{x}(\tau^+)=-\dot{x}(\tau^-)
\end{sistema}
\end{equation}
As in \cite{kunzeortega2}, by a solution we understand a function $y\in C(\mathbb{R})$ and a sequence $(t^*_n)$ of impact times such that
\begin{enumerate}
\item $\inf_n(t^*_{n+1}-t^*_n)>0$
\item $y(t^*_n)=0$ for every $n$ and $y(t)>0$ for $t\in (t^*_n,t^*_{n+1})$
\item the function $y$ is of class $C^2$ on every interval $[t^*_n,t^*_{n+1}]$ and satisfies the linear differential equation on this interval.
\item $\dot{y}(t_n^+)=-\dot{y}(t_n^-)$
\end{enumerate}
Moreover, the solution is called bounded if it also satisfies
\begin{enumerate}
\item[5.] $\sup_n (t_{n+1}-t_n)<\infty$.
\end{enumerate}
Notice that in such a case we have
$$
\sup_{t\in\mathbb{R}}|y(t)|+\esssup_{t\in\mathbb{R}}|\dot{y}(t)|<\infty.
$$
The problem can be formulated in a discrete form. We can solve the initial value problem
\begin{equation}
\begin{sistema}
\ddot{x}=-(g+\ddot{f}(t))\\
x(t_{n-1})=0,\quad \dot{x}(t_{n-1})=w_{n-1}
\end{sistema}
\end{equation}
and impose the conditions
$$
x(t_{n})=0,\quad \dot{x}(t_{n})=-w_{n}
$$
to obtain
$$
t_n=t_{n-1}+\frac{2}{g}w_{n-1}-\frac{2}{g}f[t_n,t_{n-1}]+\frac{2}{g}\dot{f}(t_{n-1})
$$
and
$$
-w_n=w_{n-1}-g(t_n-t_{n-1})-\dot{f}(t_n)+\dot{f}(t_{n-1}).
$$
where
$$
f[t_n,t_{n-1}]=\frac{f(t_n)-f(t_{n-1})}{t_n-t_{n-1}}.
$$
Substituting the first in the second we get the formulas
\begin{equation}
\begin{sistema}
t_n=t_{n-1}+\frac{2}{g}w_{n-1}-\frac{2}{g}f[t_n,t_{n-1}]+\frac{2}{g}\dot{f}(t_{n-1})
\\
w_n=w_{n-1}-2f[t_n,t_{n-1}]+\dot{f}(t_n)+\dot{f}(t_{n-1}).
\end{sistema}
\end{equation}
Inspired by such formulas we can consider the following map $S(t_0,w_0)=(t_1,w_1)$ defined by
\begin{equation}\label{mapp}
\begin{sistema}
t_1=t_{0}+\frac{2}{g}w_{0}-\frac{2}{g}f[t_1,t_{0}]+\frac{2}{g}\dot{f}(t_{0})
\\
w_1=w_{0}-2f[t_1,t_{0}]+\dot{f}(t_1)+\dot{f}(t_{0}).
\end{sistema}
\end{equation}
Notice that this is an implicit definition but we have
\begin{lemma}\label{bdef}
There exists $\bar{w}>0$, depending on $||\dot{f}||_{\infty}$, such that if $w_0>\bar{w}$ and for $t_0\in\mathbb{R}$ the map $S(t_0,w_0)=(t_1,w_1)$ is well defined and $C^3$.  
\end{lemma}
\begin{proof}
First of all notice that, since $f$ is $C^4$ and periodic
$$
t_1-t_0=\frac{2}{g}w_0+O(1)
$$
so that if $w_0\to\infty$ then $t_1-t_0\to\infty$. Now, considering the function
$$
F(t_0,t_1,w_0)=t_1-t_{0}-\frac{2}{g}w_{0}+\frac{2}{g}f[t_1,t_{0}]-\frac{2}{g}\dot{f}(t_0)
$$
we have
$$
\partial_{t_1}F(t_0,t_1,w_0)=1+\frac{2}{g}\frac{\dot{f}(t_1)(t_1-t_0)-f(t_1)+f(t_0)}{(t_1-t_0)^2}
$$
that is strictly positive for $t_1-t_0\to\infty$. So, taking $w_0$ sufficiently big, we have that for every $t_0$ we have a unique $t_1=t_1(t_0,w_0)>t_0+1$ that solves the first equation in (\ref{mapp}). Moreover applying the implicit function theorem we have, by uniqueness, that $t_1(t_0,w_0)$ is a $C^4$ function. Substituting in the second we have the thesis.
\end{proof}

It has been showed in \cite{kunzeortega2} that a good strategy to face this problem is to take a sequence $(t^*_n)$ of impact time such that $\inf_n(t^*_{n+1}-t^*_n)$ were sufficiently big in order to have a positive solution of the corresponding Dirichlet problem 
\begin{equation}
\begin{sistema}
\ddot{x}=-(g+\ddot{f}(t))\\
x(t_{n+1}^*)=x(t_n^*)=0.
\end{sistema}
\end{equation}
 Then we have to glue such solutions in a way that the elastic bounce condition holds. To this aim we have to pass to the discrete version of the problem, given by the map $S(t_0,w_0)\mapsto(t_1,w_1)$ coming from lemma \ref{bdef}. This map is not exact symplectic but $S(t_0,E_0)\mapsto(t_1,E_1)$ where $E_0:=\frac{1}{2}w_0^2$ is exact symplectic. The coordinates $(t_n,E_n)$ are conjugate, so the map can be expressed in terms of a generating function $h(t_0,t_1)$ such that
\begin{equation}\label{genfunz}
\begin{sistema}
\partial_1h(t_0,t_1)=-E_0 \\
\partial_2h(t_0,t_1)=E_1.
\end{sistema}
\end{equation}
and that can be explicitly computed giving 
\begin{equation}\label{genfun}
\begin{split}
h(t_0,t_1)=& \frac{g^2}{24}(t_1-t_0)^3+\frac{g}{2}(f(t_1)+f(t_0))(t_1-t_0)-\frac{(f(t_1)-f(t_0))^2}{2(t_1-t_0)}\\
           & -g\int^{t_1}_{t_0}f(t)dt+\frac{1}{2}\int^{t_1}_{t_0}\dot{f}^2(t)dt.
\end{split}
\end{equation}   
So the good sequence $(t^*_n)$ giving the elastic bounce condition turns to be one such that
$$
\partial_{t_1}h(t^*_{n-1},t^*_n)+\partial_{t_0}h(t^*_n,t^*_{n+1})=0.
$$
See \cite{kunzeortega2} for more details. Moreover, we can introduce an order relation between two different bouncing solution $x_1(t)$ and $x_2(t)$, saying that $x_1(t)\prec x_2(t)$ if and only if, called $(\tau_i^1)_i$ and $(\tau_i^2)_i$, the corresponding sequences of impact times, we have $\tau_i^1\leq \tau_i^2$ for every $i$.\\

\section{Existence of Aubry-Mather sets}
In this section we will prove our main result, say
\begin{teo}\label{main}
Given a $1$-periodic function $f\in C^4(\mathbb{R})$, there exists $\alpha_*$ such that for every $\alpha>\alpha_*$ there exists an orbit $(t_n^*,w_n^*)$ such that
\begin{equation}\label{condrota}
\lim_{n\to\infty}\frac{1}{n}\sum_{k=0}^{n-1}(t_{k+1}^*-t_k^*)=\alpha.
\end{equation}
\end{teo}
Our strategy of proving theorem \ref{main} relies on the classical Aubry-Mather theory for twist diffeomorphisms. To apply this theory the following hypothesis on the generating function are needed:
\begin{itemize}
\item[(H1)] $h\in C^2(\mathbb{R}^2)$
\item[(H2)] $h(t_0+1,t_1+1)=h(t_0,t_1)$ for all $(t_0,t_1)\in\mathbb{R}^2$
\item[(H3)] $\partial_{t_0t_1}h\leq\epsilon<0$ for all $(t_1,t_0)\in\mathbb{R}^2$
\end{itemize}
and we will refer to as hypothesis (H). In this case we have monotone increasing configurations $t=(t^*_n)$ that minimize the action of an exact symplectic twist diffeomorphism $S(t_0,t_1)$ and characterized by a rotation number defined as
$$
\alpha(t)=\lim_{n\to\infty}\frac{1}{n}\sum_{k=0}^{n-1}(t_{k+1}^*-t_k^*).
$$
To these configurations correspond orbits ${(t^*_n,E^*_n)}$ for the diffeomorphism that are contained in a compact invariant set called Aubry-Mather set. Precisely we have
\begin{teo}[\cite{bangert}]\label{mat}
Suppose that the generating function $h(t_0,t_1)$ of a diffeomorphisms $S(t_0,E_0)$ satisfies hypothesis (H). Then, for every $\alpha\in\mathbb{R}$ there exists an Aubry-Mather set whose orbits have rotation number $\alpha(t)=\alpha$. 
\end{teo}
%
In our case, the generating function (\ref{genfun}) does not satisfies the whole hypothesis (H), but we have
\begin{prop}\label{quasigenf}
The generating function (\ref{genfun}) belongs to $C^3(\mathbb{R}^2)$, satisfies (H2) and $\partial_{t_0t_1}h\leq\epsilon<0$ for $t_1-t_0$ sufficiently large.
\end{prop}
\begin{proof}
First of all notice that (H2) follows directly by the periodicity properties of $f$. The regularity comes from the fact that 
$$
\frac{(f(t_1)-f(t_0))^2}{(t_1-t_0)}=[\int ^1_0 \dot{f}(\lambda t_1+(1-\lambda)t_0)d\lambda ]^2(t_1-t_0).
$$
Finally, a direct calculus of the second derivative of $h$ gives
$$
\partial_{t_1t_0}h=-\frac{g^2}{4}(t_1-t_0)+O(1)\quad\mbox{as } t_1-t_0\to\infty
$$
from which we conclude. 
\end{proof}


The fact that our generating function satisfies almost all the hypothesis that we need suggests the following strategy: to look for a modification $\tilde{h}$ of the generating function $h$, that satisfies properties (H1),(H2),(H3) and that coincide with $h$ for $t_1-t_0$ sufficiently large. An idea of how to do this is presented in \cite{matherforni}. 
Before stating the result, let us recall some basic facts on the wave equation. Consider the Cauchy problem, for $k\in\mathbb{R}$
\begin{equation}\label{onde}
\begin{sistema}
u_{tt}-u_{xx}=f(t,x)\\
u(\frac{k}{2},x)=\phi(x)\\
u_t(\frac{k}{2},x)=\psi(x)
\end{sistema}
\end{equation}
It is well known that if $f\in C^1(\mathbb{R}^2)$, $\phi\in C^2(\mathbb{R})$ and $\psi\in C^1(\mathbb{R})$ then there exists a solution $u\in C^2(\mathbb{R}^2)$. Moreover, with reference to figure 1, for every $(t_*,x_*)$  
let $\Delta$ be the characteristic triangle defined by the lines $r_1: x-x_*=t-t_*$, $r_2: x-x_*=t_*-t$, $r_3:t=\frac{k}{2}$ and let $J$ be the segment on $r_3$ defined by $r_1$ and $r_2$. Now, if $f_1=f_2$ on $\Delta$, $\phi_1=\phi_2$ on $J$ and $\psi_1=\psi_2$ on $J$ then we have that $u_1(t_*,x_*)=u_2(t_*,x_*)$, where $u_1$ and $u_2$ are the corresponding solutions. Finally, if the periodic conditions $f(t,x+1)=f(t,x)$, $\phi(x+1)=\phi(x)$ and $\psi(x+1)=\psi(x)$ are imposed then we have that the solution satisfies $u(t,x+1)=u(t,x)$.
\begin{prop}\label{modif}
Consider a function $h:\mathbb{R}^2\rightarrow\mathbb{R}$, $h\in C^3(\mathbb{R}^2)$, satisfying property (H2) and such that $\partial_{t_0t_1}h\leq\epsilon<0$ for $t_1-t_0$ sufficiently large. Suppose 
that 
$\partial_{t_0t_1t_1}h$ and $\partial_{t_0t_0t_1}h$ are bounded. 
Then there exists a function $\tilde{h}:\mathbb{R}^2\rightarrow\mathbb{R}$ that satisfies property (H1),(H2),(H3) and such that it coincides with $h$ for $t_1-t_0$ sufficiently large.
\end{prop} 
\begin{proof}
Let us start performing the change of variables $t_0=x-t, t_1=x+t$ in (\ref{onde}). We get
\begin{equation}
\begin{sistema}
u_{t_0t_1}=\tilde{f}(t_0,t_1)\\
u(t_0,t_0+k)=\tilde{\phi}(t_0)\\
(u_{t_1}-u_{t_0})(t_0,t_0+k)=\tilde{\psi}(t_0)
\end{sistema}
\end{equation}
where $\tilde{f}(t_0,t_1)=-\frac{1}{4}f(\frac{t_1-t_0}{2},\frac{t_1+t_0}{2})$, $\tilde{\phi}(t_0)=\phi(t_0+\frac{k}{2})$ and $\tilde{\psi}(t_0)=\psi(t_0+\frac{k}{2})$. In the new variables the characteristic triangle is defined, for every $(t_0^*,t_1^*)$, by the lines $\tilde{r}_1: t_0=t_0^*$, $\tilde{r}_2: t_1=t_1^*$, $\tilde{r}_3:t_1=t_0+k$. The change of variables does not affect the regularity so we have a solution $u\in C^2(\mathbb{R}^2)$ that, in an analogous sense as in (\ref{onde}), is unique on every characteristic triangle. Moreover, if $\tilde{f}(t_0,t_1)=\tilde{f}(t_0+1,t_1+1)$, $\tilde{\phi}(t_0)=\tilde{\phi}(t_0+1)$ and $\tilde{\psi}(t_0)=\tilde{\psi}(t_0+1)$ then, by the change of variable, we have that the solution satisfies $\tilde{u}(t_0,t_1)=\tilde{u}(t_0+1,t_1+1)$.
\begin{figure}[h]\label{caratt}
\begin{center}
\includegraphics[scale=0.6]{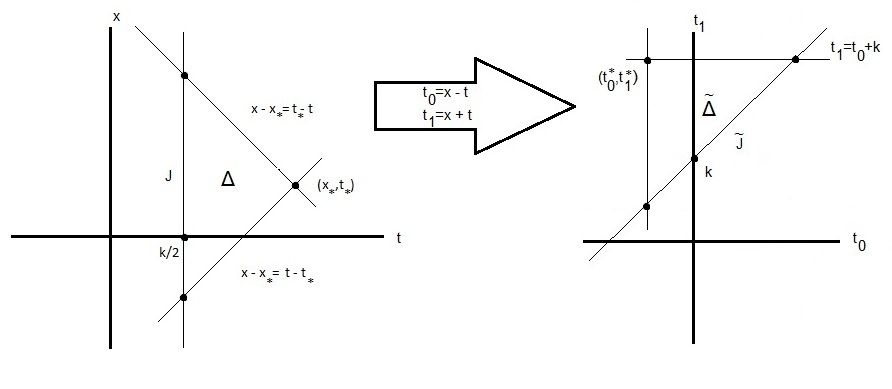}
\end{center}
 \caption{The characteristic triangle}
 \end{figure}  
\\To complete our preliminaries, consider, for any constant $A>0$, $C\in\mathbb{R}$, $x_0\in\mathbb{R}$, the function
$$
\phi(x)=\frac{A+C(x-x_0)}{(x-x_0)^2}
$$
and notice that $\phi(x)$ is bounded from below and
$$
\inf_{x<x_0}\phi(x)\leq 0.
$$ 
Now we can start. Let us call $d(t_0,t_1)=\partial_{t_1t_0}^2 h(t_0,t_1)$. 
By hypothesis we have that there exists $k>0$ sufficiently large and $\epsilon<0$ such that, if $t_1-t_0\geq k$ then
\begin{equation}\label{fre}
d(t_0,t_1)\leq\epsilon<0.
\end{equation}
Now choose $\tilde{\epsilon}$ such that $\epsilon<\tilde{\epsilon}<0$. In the definition of $\phi$ let $$A=\tilde{\epsilon}-\epsilon,\quad C=\frac{1}{2}||\partial_{t_1}d-\partial_{t_0}d||_\infty, \quad x_0=k.$$
Notice that $A$ is positive due to the choice of $\tilde{\epsilon}$ and that $C$ is finite due to the assumption on $\partial_{t_0t_1t_1}h$ and $\partial_{t_0t_0t_1}h$. So we can define $$I:=\inf_{x<k}\phi(x).$$
As we said we have $-\infty<I\leq 0$ so that we can fix $H$ such that $H<I$.
So let
\begin{equation}
\begin{split}
D(t_0,t_1)=&d(\frac{t_0+t_1-k}{2},\frac{t_0+t_1+k}{2})+\\ &\frac{1}{2}(\partial_{t_1}d(t_0,t_1)-\partial_{t_0}d(t_0,t_1))(t_1-t_0-k)+H(t_1-t_0-k)^2
\end{split}
\end{equation}
and define
\begin{equation}
\tilde{d}(t_0,t_1)=
\begin{sistema}
d(t_0,t_1) \quad\mbox{if } t_1-t_0\geq k\\
D (t_0,t_1) \quad\mbox{if } t_1-t_0< k
\end{sistema}
\end{equation}


It is easily seen that $\tilde{d}(t_0+1,t_1+1)=\tilde{d}(t_0,t_1)$. Moreover we claim that $\tilde{d}\in C^1(\mathbb{R}^2)$. Indeed clearly, it comes from the regularity of $h$ that every piece of the definition is $C^1$. Moreover it is also immediate that 
  $$
  d(t_0,t_0+k)=D (t_0,t_0+k),
  $$  
and a long but straight computation of the partial derivatives gives the requested regularity.
Now let us study how to satisfy property (H3).
We claim that,    
\begin{equation}\label{gre}
\tilde{d}(t_0,t_1)\leq \tilde{\epsilon}<0
\end{equation}
where $\tilde{\epsilon}$ comes from the definition of $H$.
Indeed, it is clear by hypothesis for $t_1-t_0\geq k$. For $t_1-t_0<k$, noticing that $\frac{t_0+t_1+k}{2}-\frac{t_0+t_1-k}{2}=k$ and remembering the definition of $H$ and the definition of $I$ as an infimum we have:
\begin{equation}
\begin{split}
&D(t_0,t_1)\leq \epsilon + C|t_1-t_0-k|+I(t_1-t_0-k)^2\leq \\
& \epsilon - C(t_1-t_0-k)+\frac{\tilde{\epsilon}-\epsilon+C(t_1-t_0-k)}{(t_1-t_0-k)^2}(t_1-t_0-k)^2=\tilde{\epsilon}<0
\end{split}
\end{equation} 

Now consider the Cauchy problem
\begin{equation}
\begin{sistema}
u_{t_0t_1}=\tilde{d}(t_0,t_1)\\
u(t_0,t_0+k)=h(t_0,t_0+k)=:\tilde{\phi}(t_0)\in C^2\\
(u_{t_1}-u_{t_0})(t_0,t_0+k)=(h_{t_1}-h_{t_0})(t_0,t_0+k)=:\tilde{\psi}(t_0)\in C^1.
\end{sistema}
\end{equation} 
By what we discussed in the preliminaries, we have a solution $\tilde{h}\in C^2(\mathbb{R}^2)$ that is unique on every characteristic triangle, such that $\tilde{h}(t_0,t_1)=\tilde{h}(t_0+1,t_1+1)$ and that, by construction, satisfies (H3). Finally, since $\tilde{d}=d$ if $t_1>t_0+k$, we have by uniqueness that $\tilde{h}=h$ if $t_1>t_0+k$, so that $\tilde{h}$ satisfies the thesis. 
\end{proof}

Now we are ready for the
\begin{proof}[Proof of theorem \ref{main}]
First of all notice that Proposition \ref{quasigenf} guarantees most of the hypothesis required by Proposition \ref{modif}. To verify the boundedness of $\partial_{t_0t_1t_1}h$ and $\partial_{t_0t_0t_1}h$ consider, first of all, in (\ref{genfun}), the term 
$$
\frac{(f(t_1)-f(t_0))^2}{2(t_1-t_0)}.
$$
A direct computation gives the boundedness of the third derivatives for $t_1-t_0$ large. If $t_1-t_0$ is small we just have to remember that
$$
\frac{(f(t_1)-f(t_0))^2}{(t_1-t_0)}=[\int ^1_0 \dot{f}(\lambda t_1+(1-\lambda)t_0)d\lambda ]^2(t_1-t_0)
$$
and perform a direct computation. The boundedness of the third derivatives is trivial for the other terms remembering the regularity and the periodicity of $f$. So we can apply Proposition \ref{modif} to (\ref{genfun}) to have a generating function $\tilde{h}(t_0,t_1)$ to which we can apply theorem \ref{mat}. Using the terminology of Mather \cite{matherams} we can find for every $\alpha\in\mathbb{R}$ a minimal configuration $t=(t^*_n)_{n\in\mathbb{Z}}$ with rotation number $\alpha(t)=\alpha$. Moreover, for this configuration we have that 
\begin{equation}\label{stat}
 \partial_2\tilde{h}(t_{n-1}^*,t^*_n)+\partial_1\tilde{h}(t^*_n,t^*_{n+1})=0
 \end{equation}
and
\begin{equation}\label{rot}
 |t^*_n-t^*_0-n\alpha(t)|<1.
 \end{equation}
Furthermore $\tilde{h}$ generates a diffeomorphism $\tilde{S}$ in the sense that 
 \begin{equation}
 \begin{sistema}
 \partial_1\tilde{h}(t_0,t_1)=-E_0 \\
 \partial_2\tilde{h}(t_0,t_1)=E_1
 \end{sistema}
\Leftrightarrow \tilde{S}(t_0,E_0)=(t_1,E_1), 
 \end{equation}
so, letting $E^*_n=\partial_2\tilde{h}(t^*_{n-1},t^*_n)=-\partial_1\tilde{h}(t^*_n,t^*_{n+1})$, we have that $(t^*_n,E^*_n)$ is a complete orbit of $\tilde{S}$: we call it minimal orbit. Yet, from (\ref{rot}) we have
 $$
 t^*_0+n\alpha(t)-1<t^*_n<t^*_0+n\alpha(t)+1
 $$
 and
 $$
 t^*_0+n\alpha(t)+\alpha(t)-1<t^*_{n+1}<t^*_0+n\alpha(t)+\alpha(t)+1
 $$
 from which we have
 \begin{equation}\label{rot2}
 \alpha(t)-2<t^*_{n+1}-t^*_n<\alpha(t)+2\quad\mbox{for every }n.
 \end{equation}
It  means that there exists $\alpha_*$ such that if $\alpha>\alpha_*$ we have that for every $n$ 
\begin{equation}\label{gran}
t^*_{n+1}-t^*_n>k.
\end{equation} 
where $k>0$ is a large positive constant such that $h(t_0,t_1)=\tilde{h}(t_0,t_1)$ if $t_1-t_0>k$.
Finally we claim that the orbit $(t^*_n,E^*_n)$ of $\tilde{S}$ is actually an orbit of $S$. Indeed we have, remembering (\ref{gran}), that
\begin{equation}
S(t^*_n,t^*_{n+1})=
\begin{sistema}
-\partial_1 h(t^*_n,t^*_{n+1}) \\
 \partial_2 h(t^*_n,t^*_{n+1})
\end{sistema}
=
\begin{sistema}
-\partial_1\tilde{h}(t^*_n,t^*_{n+1}) \\
 \partial_2\tilde{h}(t^*_n,t^*_{n+1})
\end{sistema}
=(E_n^*,E^*_{n+1}).
\end{equation}

\end{proof}

So, coming back to the physical problem of the bouncing ball 
\begin{cor}\label{cr}
There exists $\alpha_*$ such that for every $\alpha>\alpha_*$ then there exists a bouncing solution such that
	 $$
 \lim_{n\to\infty}\frac{1}{n}\sum_{k=0}^{n-1}(t_{k+1}^*-t_k^*)=\alpha.
 $$ 
 Moreover,
	 
\begin{itemize}
\item If $\alpha=p/q$ is rational then 
	\begin{itemize}
	\item there exists a $p$-periodic bouncing solutions of (\ref{prob}) with $q$ bounces in a period,
	\item if there exist two different periodic solutions $x^1(t)$ and $x^2(t)$, $x^1(t)\prec x^2(t)$ with the same rotation number $\alpha$ such that there is not another periodic solution $x^*(t)$ with the same rotation number such that $x^1(t)\prec x^*(t)\prec x^2(t)$ then there exist two different solutions $x^+(t)$ and $x^-(t)$ with rotation number $\alpha$ such that the corresponding sequences of impact times satisfy
	$$
	\lim_{i\to-\infty}|t^+_i-t^1_i|=0\quad\mbox{and}\quad\lim_{i\to\infty}|t^+_i-t^2_i|=0
	$$
	and
	$$
	\lim_{i\to-\infty}|t^-_i-t^2_i|=0\quad\mbox{and}\quad\lim_{i\to\infty}|t^-_i-t^1_i|=0
	$$ 
	\end{itemize}

\item If $\alpha$ is irrational then the sequence of impact times $t_i^*$ of the solution $x^*(t)$ with rotation number $\alpha$ is such that the set $\{t_i^*+\mathbb{Z},i\in\mathbb{Z}\}$ is either dense in $\mathbb{R}/\mathbb{Z}$ or a Cantor set in $\mathbb{R}/\mathbb{Z}$. 
	\end{itemize}
\end{cor}

 \begin{proof}
 Condition (\ref{stat}) is the one that guaranties the condition of elastic bouncing. So we have that to every minimal orbit of $S$ with rotation number $\alpha>\alpha_*$ corresponds a bouncing solution of problem (\ref{prob}) such that
 $$
 \lim_{n\to\infty}\frac{1}{n}\sum_{k=0}^{n-1}(t^*_{k+1}-t^*_k)=\alpha
 $$   
where $t^*_n$ represent the time of the $n$-th bounce on the racket. The thesis follows directly from the general Aubry-Mather theory, remembering that in the case of a rational rotation number the periodic minimal orbits $(t^*_n,E^*_n)$ for $S$ are such that
	\begin{equation}
	\begin{sistema}
	t^*_{n+q}-t^*_n=p\\
	E^*_{n+q}=E^*_n
	\end{sistema}
	\end{equation}  
 \end{proof}
 
 \begin{remark}
 Notice that we can interpret the rotation number of a bouncing solution as an average of the distance between two consecutive impact times.
 \end{remark}

\noindent \textbf{Acknowledgements.} I wish to thank Professor Rafael Ortega for having introduced me to this topic and for his constant supervision.

\bibliographystyle{plain}
\bibliography{biblio4}

\end{document}